\documentclass[a4paper,10pt,leqno,oneside,draft]{amsart}
\usepackage{amsmath}
\usepackage{amsfonts}
\usepackage{enumerate}
\usepackage{amsthm}
\usepackage{bm}
\usepackage{mathtools}
\usepackage{amssymb}
\usepackage[top=20truemm, bottom=20truemm, left=15truemm, right=15truemm]{geometry}
\usepackage{enumitem}
\usepackage{color}
\makeatletter \@addtoreset{equation}{section} \makeatother
\newcommand{\eref}[1]{(\ref{#1})}
\newcommand{\tref}[1]{Theorem \ref{#1}}
\newcommand{\pref}[1]{Proposition \ref{#1}}
\newcommand{\cref}[1]{Corollary \ref{#1}}
\newcommand{\lref}[1]{Lemma \ref{#1}}
\newcommand{\rref}[1]{Remark \ref{#1}}
\newcommand{\sref}[1]{Section \ref{#1}}

\theoremstyle{plain} \newtheorem{thm}{Theorem}[section] \newtheorem{lem}{Lemma}[section] \newtheorem{prop}{Proposition}[section] \newtheorem{cor}{Corollary}[section]
\theoremstyle{definition} \newtheorem{rem}{Remark}[section] 
\title[Navier--Stokes variational inequalities]{Local-in-time strong solvability of Navier--Stokes type\\ variational inequalities by Rothe's method}
\author[Takahito Kashiwabara]{Takahito Kashiwabara}
\address{Graduate School of Mathematical Sciences, The University of Tokyo, 3-8-1 Komaba, Meguro, 153-8914 Tokyo, Japan}
\email{tkashiwa@ms.u-tokyo.ac.jp}
\thanks{This work was supported by JSPS KAKENHI Grant-in-Aid for Scientific Research(C), Grant Number 24K06860.}

\begin{document}
\begin{abstract}
	We consider parabolic variational inequalities in a Hilbert space $V$, which have a non-monotone nonlinearity of Navier--Stokes type represented by a bilinear operator $B: V \times V \to V'$ and a monotone type nonlinearity described by a convex, proper, and lower-semicontinuous functional $\varphi : V \to (-\infty, +\infty]$.
	Existence and uniqueness of a local-in-time strong solution in a maximal-$L^2$-regularity class and in a Kiselev--Ladyzhenskaya class are proved by discretization in time (also known as Rothe's method), provided that a corresponding stationary Stokes problem admits a regularity structure better than $V$ (which is typically $H^2$-regularity in case of the Navier--Stokes equations).
	Since we do not assume the cancelation property $\left< B(u, v), v \right> = 0$, in applications we may allow for broader boundary conditions than those treated by the existing literature.
\end{abstract}
\maketitle

\section{Introduction}
Let $V \subset H = H' \subset V'$ be a Gelfand triple of Hilbert spaces, with their norms $\|\cdot\|_H$ and $\|\cdot\|_V$, in which the embeddings are continuous and dense.
We assume that the first embedding $V \hookrightarrow H$ is compact, and denote the inner product of $H$ and the duality pairing between $V'$ and $V$ by $(\cdot, \cdot)$ and $\left< \cdot, \cdot \right>$, respectively.

Let $a : V \times V \to \mathbb R$ be a bounded and coercive bilinear form, i.e., $a(u, u) \ge \alpha \|u\|_V^2$ for all $u \in V$ and some $\alpha > 0$.
Since the norm of $V$ is equivalent to $\sqrt{a(u, u)}$, we redefine it as $\|u\|_V := \sqrt{a(u, u)}$ in what follows.
It can be identified with a bounded linear operator $A : V \to V'; \; u \mapsto (v \mapsto a(u, v))$.

We introduce a ``Navier--Stokes type nonlinearity'' as a bounded bilinear operator $B : V \times V \to V'$, which satisfies
\begin{equation} \label{eq: boundedness of B}
	|\left< B(u, v), w \right>| \le C_B \|u\|_V \|v\|_V \|w\|_V \qquad \forall u, v, w \in V
\end{equation}
for some constant $C_B > 0$.
We further introduce a ``monotone type nonlinearity'' as a convex, proper, and lower-semicontinuous functional $\varphi : V \to (-\infty, +\infty]$.
By the Hahn--Banach theorem, there exists a constant $C_{\varphi 1} > 0$ such that (cf.\ \cite[p.\ 161]{Bre1972})
\begin{equation} \label{eq: phi is bounded from below}
	\varphi(v) \ge -C_{\varphi 1}(\|v\|_V + 1) \qquad \forall v \in V.
\end{equation}

In this setting, we consider the following parabolic variational inequality: given an external force $f : [0, T] \to H$ and initial datum $u^0 \in V$, find $u : [0, T] \to V$ such that $u(0) = u^0$, and, for a.e.\ $t \in (0, T)$,
\begin{equation} \label{eq: VI}
	(\partial_t u(t), v - u(t)) + a(u(t), v - u(t)) + \left< B(u(t), u(t)), v - u(t) \right> + \varphi(v) - \varphi(u(t)) \ge (f(t), v - u(t)) \qquad \forall v \in V.
\end{equation}

Such a parabolic variational inequality of the Navier--Stokes type was already examined in the framework of nonlinear semigroups by \cite{Bre1972}.
However, it assumed the two-dimensional situation.
In fact, let $\Omega \subset \mathbb R^d \, (d = 2, 3)$ be a bounded smooth domain and set, as an example, $H := L^2_\sigma(\Omega)$, $V := H^1_{0, \sigma}(\Omega)$ (cf.\ \cite[p.\ 250]{Bar2010}), $a(u, v) := (\nabla u, \nabla v)$, and $B(u, v) := (u \cdot \nabla) v$ to treat the standard incompressible Navier--Stokes equations with the homogeneous Dirichlet boundary condition on $\partial\Omega$.
Then assumption (3) of \cite[p.\ 159]{Bre1972} reads
\begin{equation*}
	(\underbrace{ u \cdot \nabla u - v \cdot \nabla v }_{ = (u - v) \cdot \nabla u + v \cdot \nabla (u - v) }, u - v) \ge - C \|u - v\|_{H^1(\Omega)} \|u - v\|_{L^2(\Omega)} \|u\|_{H^1(\Omega)}.
\end{equation*}
We can expect it to be valid only when both of the 2D Ladyzhenskaya inequality $\|u - v\|_{L^4(\Omega)}^2 \le C \|u - v\|_{L^2(\Omega)} \|u - v\|_{H^1(\Omega)}$ and the cancelation property $(v \cdot \nabla (u - v), u - v) = 0$ hold.
We remark that, for the latter equality to hold, we need---in addition to $\operatorname{div} v = 0$---the impermeability condition $v \cdot \nu = 0$ on the boundary, where $\nu$ is the outer unit normal to $\partial\Omega$.

A similar variational inequality was also considered by \cite{Kon2000}, which proved existence of a weak solution for $d = 3$ and unique existence of a strong solution for $d = 2$ (assumption (2.19) of \cite[p.\ 881]{Kon2000} holds only for $d = 2$ by a similar reason as above) by the Galerkin method.
Analysis of related problems is also found in \cite{Che2001, Mig2013, San2023},
We emphasize that the cancelation property $\left< B(u, v), v \right> = 0$ played an essential role in these results.
\begin{rem}
	The impermeability condition was not imposed in the problems considered by \cite{Che2001, Kon2000, Mig2013, San2023}.
	Instead, $u_\tau := u - (u \cdot \nu) \nu = 0$ and a unilateral condition on the total pressure $P := p + \frac{|u|^2}{2}$ ($p$ is the usual pressure) are supplemented.
	If we employ $a(u, v) := (\operatorname{rot} u, \operatorname{rot} v)$ and $B(u, v) := \operatorname{rot} u \times v$, then we see that $-\Delta u + \nabla p + u\cdot \nabla u = \operatorname{rot} \operatorname{rot} u + \nabla P + B(u, u)$ if $\operatorname{div} u = 0$,
	that $(\operatorname{rot} \operatorname{rot} u, v) = a(u, v)$ if $v_\tau = 0$ on $\partial\Omega$, and that $\left< B(u, v), v \right> = 0$.
	However, in case one would like to adopt a boundary condition that does not prescribe $P$, the impermeability condition would be needed to ensure the cancelation property $(u \cdot \nabla v, v) = 0$.
	In more precise, a slightly weaker condition $(u \cdot \nabla v, v) \ge 0$ is sufficient to make an energy estimate available, which holds if the outflow condition $u \cdot \nu \ge 0$ is imposed on the boundary (see \cite[p.\ 139]{AKM1989} and \cite{ZhSa2016}).
\end{rem}

Strong solvability of a more general evolution equation, in which both of the monotone part and the non-monotone perturbation part can depend on time and can be multi-valued, was established by \cite{Ota1982}.
As a concrete application, existence and uniqueness of a local-in-time strong solution such that $u \in L^2(0, T_*; H^2(\Omega))$ for some $T_*$ are shown for the 3D Navier--Stokes equations ($\Omega$ can even depend on $t$) using the cancelation property (see \cite[p.\ 295]{Ota1982}).

An example of boundary conditions in which the cancelation property of $B$ does not necessarily hold is the leak boundary condition of friction type (LBCF) introduced by \cite{Fuj1994} (nevertheless, \cite{Fuj1994} considered only the stationary case; see also \cite{Fuj2001, Fuj2002} for the non-stationary Stokes case).
In \cite{Kas2013JDE}, existence and uniqueness of a local-in-time strong solution such that $u \in W^{1,\infty}(0, T_*; L^2(\Omega))$ were shown, but only under the assumption that the initial permeability $u^0 \cdot \nu$ is sufficiently small on the boundary.
This additional assumption arose in order to compensate the lack of the cancelation property.
Another boundary condition, in which the outflow condition $u \cdot \nu \ge 0$ is imposed in a way similar to the Signorini condition in elasticity problems, was analyzed in \cite[Section III.3]{AKM1989} and independently in \cite{ZhSa2016}.
We refer to \cite[Chapter 6]{KiCa2021} for a treatment of various mixed boundary conditions including these subdifferential boundary conditions.
We remark that $H^2(\Omega)$-regularity was not obtained in these results, cf.\ \cite[p.\ 777]{Kas2013JDE}.

\begin{rem}
	For the linear leak boundary condition $u_\tau = 0$, $\sigma(u, p) \nu \cdot \nu = 0$, or the so-called do-nothing boundary condition $\sigma(u, p) \nu = 0$, where $\sigma(u, p)$ denotes the stress tensor, one can prove strong solvability such that $u \in L^2(0, T_*; H^2(\Omega))$, utilizing the eigenfunctions of a  Stokes operator in the Galerkin method (see \cite[Theorem 6]{HRT1996}).
	This strategy, however, is not applicable to e.g.\ LBCF because the associated Stokes operator is no longer linear.
\end{rem}

In view of the situation mentioned above, the goal of the present work is to establish local-in-time strong solvability of the Navier--Stokes type variational inequality \eref{eq: VI}
\begin{enumerate}[label=\arabic*), leftmargin=*]
	\item with $d = 3$ included;
	\item under a weaker condition on $B$ than the cancelation property to allow for boundary conditions like LBCF;
	\item with higher regularity than $V$ if a corresponding stationary Stokes problem admits a suitable regularity structure.
\end{enumerate}
Since it is non-trivial to achieve 1)--3) above by the Galerkin method, which somewhat corresponds to spatial discretization, we will make use of temporal discretization (also known as Rothe's method).
It was already exploited in \cite{FuKa2026} to prove $H^2$-regularity for a non-stationary Bingham problem, in which a fully-implicit backward Euler scheme was combined with the truncation technique of $B$ (cf.\ \cite[Section 5.7]{Bar2010}).
To the contrary, in this paper we propose a semi-implicit scheme where the convection velocity is taken from the previous time step, not truncating $B$.
Then, at the stage of deriving \emph{a priori} estimates for discrete solutions, we already need to restrict to some $T_* < T$.
To the best of our knowledge, such a local-in-time argument in Rothe's method seems to be new and is of independent interest.

The organization of this paper is as follows.
In \sref{sec2}, we explain the detailed assumptions on $H, V, B$, and $\varphi$ and state the main results.
Examples of boundary value problems that fit the main results are described in \sref{sec: examples}.
In \sref{sec3}, we discuss solvability of a stationary Oseen type problem, which is necessary to solve our semi-implicit scheme in \sref{sec4}.
We also show \emph{a priori} estimates for the discrete-in-time solutions that are uniform in the time increment $\Delta t$.
Passage to the limit $\Delta t \to 0$ is discussed in \sref{sec5} and we construct a strong solution $u$ of \eref{eq: VI} such that $u \in H^1(0, T; H)$ (which is similar to a maximal $L^2$-regularity class).
In \sref{sec6}, under additional assumptions on $u^0$ and $f$, we prove that $u \in W^{1,\infty}(0, T; H)$ (which is similar to a Kiselev--Ladyzhenskaya class).
The argument of \sref{sec6} requires some facts of evolution equations involving maximal monotone operators, proofs of which are provided in the appendix for the sake of completeness.

\section{Main results} \label{sec2}
Let us clarify the detailed hypotheses on the Navier--Stokes type nonlinearity $B$ and the monotone type nonlinearity $\varphi$.
For each assumption, we will explain how it is satisfied for the 3D incompressible Navier--Stokes equations (abbreviated as NS below) in a bounded smooth domain $\Omega \subset \mathbb R^d \, (d = 3)$, in which case $H$ and $V$ are chosen to be closed subspaces of $\{ v \in L^2(\Omega)^d \mid \operatorname{div} v = 0 \text{ in } \Omega \}$ and $\{ v \in H^1(\Omega)^d \mid \operatorname{div} v = 0 \text{ in } \Omega \}$, respectively (their actual choices depend on boundary conditions).

First, we relax the cancelation property $\left< B(u, v), v \right> = 0$ as follows.

\textbf{(H1)} There exist $C > 0$ and $\beta \in (0, 1]$ such that, for arbitrary $u \in V$ and $v \in V$,
\begin{equation*}
	|\left< B(u, v), v \right>| \le C \|u\|_V \|v\|_V \|v\|_H^\beta \|v\|_V^{1 - \beta}.
\end{equation*}
By Young's inequality for real numbers, this yields
\begin{align*}
	|\left< B(u, v), v \right>| \le C \big( \|u\|_V^{2/\beta} \|v\|_H^2 \big)^{\beta/2} \cdot \big( \|v\|_V^2 \big)^{1 - \beta/2} \le C_{\theta_1, \epsilon} \|u\|_V^{\theta_1} \|v\|_H^2 + \epsilon \|v\|_V^2 \qquad \forall \epsilon > 0,
\end{align*}
where we have rephrased $\theta_1 := 2/\beta \ge 2$.
In NS case, this holds with $\theta_1 = 4 \, (\beta = 1/2)$; in fact we have
\begin{equation*}
	|(u \cdot \nabla v, v)| \le \|u\|_{L^6(\Omega)} \|\nabla v\|_{L^2(\Omega)} \|v\|_{L^3(\Omega)} \le C \|u\|_{H^1(\Omega)} \|v\|_{H^1(\Omega)} \|v\|_{L^2(\Omega)}^{1/2} \|v\|_{H^1(\Omega)}^{1/2} \quad \forall u, v \in V,
\end{equation*}
as a result of the Sobolev and Gagliardo--Nirenberg inequalities.

Next we assume a kind of continuity of $B$ with respect to weak convergence in $V$.

\textbf{(H2)} For each fixed $u \in V$, if $v_j \rightharpoonup v$ weakly in $V$ and $w_j \rightharpoonup w$ weakly in $V$ as $j \to \infty$, then $\left< B(u, v_j), w_j \right> \to \left< B(u, v), w \right>$.
In particular, when $v_j \rightharpoonup v$ weakly in $V$, we have
\begin{equation*}
	\lim_{j \to \infty} \left< B(u, v_j), v_j - w \right> = \left< B(u, v), v - w \right> \quad \forall w \in V,
\end{equation*}
so that $B(u, \cdot): V \to V'$ is pseudo-monotone.
In NS case, (H2) follows from
\begin{equation*}
	|(u \cdot \nabla v_j, v_j - w) - (u \cdot \nabla v, v - w)| \le |(u \cdot \nabla v_j, v_j - v)| + |(u \cdot \nabla (v_j - v), v - w)| \to 0 \quad (j \to \infty),
\end{equation*}
where we have used the H\"older inequality with $\frac14 + \frac12 + \frac14 = 1$ and the compact embedding $H^1(\Omega) \hookrightarrow L^4(\Omega)$.
\begin{rem}
	Recalling the precise definition of the pseudo-monotonicity (see \cite[D\'efinition II.2.1]{Lio1969}), in addition to the weak convergence $v_j \rightharpoonup v$ one could have assumed that $\limsup_{j \to \infty} \left< B(u, v_j), v_j - v \right> \le 0$, which was actually unused in the above argument.
	In this sense, (H2) is a stronger concept than the pseudo-monotonicity.
\end{rem}

In the following we presume existence of a Hilbert space $W \subset V$ whose embedding is dense and compact.
In NS case, $W$ is taken to be a closed subspace of $\{v \in H^2(\Omega)^d \mid \operatorname{div} v = 0 \text{ in } \Omega \}$ depending on boundary conditions.
Further assumptions on $W$ in connection with $B$ are described as follows.

\textbf{(H3)} There exists $C > 0$ such that, if $u \in W$ and $v \in V$ then $B(u, v) \in H$ and
\begin{equation*}
	\|B(u, v)\|_H \le C \|u\|_W \|v\|_V.
\end{equation*}	
In NS case, this follows from $\|u \cdot \nabla v\|_{L^2(\Omega)} \le \|u\|_{L^\infty(\Omega)} \|\nabla v\|_{L^2(\Omega)}$ and the embedding $H^2(\Omega) \hookrightarrow L^\infty(\Omega)$.

\textbf{(H4)} There exist $C > 0$ and $\gamma \in (0, 1]$ such that, if $u \in V$ and $v \in W$ then $B(u, v) \in H$ and
\begin{equation*}
	\|B(u, v)\|_H \le C \|u\|_V \|v\|_V^\gamma \|v\|_W^{1 - \gamma} \qquad (\text{in particular, } \|B(u, v)\|_H \le C \|u\|_V \|v\|_W).
\end{equation*}
By Young's inequality, this yields
\begin{equation*}
	\|B(u, v)\|_H \le C_{\theta_2, \epsilon} \|u\|_V^{\theta_2} \|v\|_V + \epsilon \|v\|_W \qquad \forall \epsilon > 0,
\end{equation*}
where we have rephrased $\theta_2 := 1/\gamma \ge 1$.
In NS case, (H4) holds with $\theta_2 = 2 \, (\gamma = 1/2)$ because $\|u \cdot \nabla v\|_{L^2(\Omega)} \le \|u\|_{L^6(\Omega)} \|\nabla v\|_{L^3(\Omega)} \le C \|u\|_{H^1(\Omega)} \|v\|_{H^1(\Omega)}^{1/2} \|v\|_{H^2(\Omega)}^{1/2}$.

Finally, let us make an essential hypothesis on $W$ in connection with $\varphi$, which describes an elliptic regularity structure for a stationary Stokes version of \eref{eq: VI}.
To state it, we recall the definition of the effective domain $D(\varphi) := \{ v \in V \mid \varphi(v) < \infty \}$ and that of the subdifferential of $\varphi$.
We define a multi-valued operator $\partial\varphi : V \to V'$, which can be identified with either a map $\partial\varphi : V \to 2^{V'}$ or a subset $\partial\varphi \subset V \times V'$, by $\partial\varphi(u) = \{ \xi \in V' \mid \left< \xi, v - u \right> \le \varphi(v) - \varphi(u) \; \forall v \in V \}$ together with $D(\partial\varphi) := \{ u \in V \mid \partial\varphi(u) \neq \emptyset \}$.

\textbf{(H5)} Let $u \in D(\varphi)$ and $f \in H$ satisfy the variational inequality
\begin{equation*}
	a(u, v - u) + \varphi(v) - \varphi(u) \ge \left< f, v - u \right> \qquad \forall v \in V,
\end{equation*}
that is, $f \in (A + \partial\varphi)(u)$.
Then we have $u \in W \cap D(\partial\varphi)$ and
\begin{equation*}
	\|u\|_W \le C_{\mathrm{reg}} \|f\|_H + C_{\varphi 2}
\end{equation*}
for some positive constants $C_{\mathrm{reg}}, C_{\varphi 2}$ independent of $u$ and $f$.
In other words, $(A + \partial\varphi)^{-1}(H) \subset W$ holds.
Examples of variational inequalities that satisfy (H5) will be mentioned in \sref{sec: examples}.

\begin{rem}
	A subdifferential of $\varphi$, whose domain and image are enlarged to $H$ and restricted to $H$ respectively, is defined by $\partial_* \varphi : H \to 2^{H}$; $u \mapsto \{ \xi \in H \mid (\xi, v - u) \le \varphi(v) - \varphi(u) \; \forall v \in H \}$ with $D(\partial_* \varphi) := \{ u \in H \mid \partial_* \varphi(u) \neq \emptyset \}$, where we note that $\varphi|_{H \setminus V}$ is set to be $+ \infty$.
	We see that $\partial_* \varphi(u) = \partial\varphi(u) \cap H$ for all $u \in V$.
\end{rem}

We are ready to state our main results.

\begin{thm} \label{main thm1}
	Under the hypotheses of (H1)--(H5) above, let $f \in L^2(0, T; H)$ and $u^0 \in D(\varphi) \subset V$ be arbitrary.
	Then there exist some $T_* \in (0, T]$ and a unique $u \in H^1(0, T_*; H) \cap L^\infty(0, T_*; V) \cap L^2(0, T_*; W)$ that solves \eref{eq: VI} for a.e.\ $t \in (0, T)$.
\end{thm}

\begin{thm} \label{main thm2}
	Under the hypotheses of (H1)--(H5) above, let $f \in W^{1,1}(0, T; H)$ and $u^0 \in W \cap D(\partial\varphi)$ with $(A + \partial\varphi)(u^0) \cap H \neq \emptyset$.
	Then there exist some $T_* \in (0, T]$ and a unique solution $u \in W^{1, \infty}(0, T_*; H) \cap H^1(0, T_*; V) \cap L^\infty(0, T_*; W)$ of \eref{eq: VI}.
	Moreover, \eref{eq: VI} holds for all $t \in [0, T_*)$, with $\partial_t u(t)$ replaced by the right-derivative $\frac{d^+}{dt} u(t) := \lim_{h \downarrow 0} \frac{u(t + h) - u(t)}{h}$.
\end{thm}

\begin{rem}\label{rem: after main thm}
	(i) Slightly modifying the proofs of Theorems \ref{main thm1} and \ref{main thm2}, we see that the same conclusions hold for the problem in which \eref{eq: VI} is replaced by
	\begin{equation*}
		(\partial_t u(t), v - u(t)) + a(u(t), v - u(t)) + \left< B(b, u(t)), v - u(t) \right> + \varphi(v) - \varphi(u(t)) \ge (f(t), v - u(t)) \qquad \forall v \in V,
	\end{equation*}
	where $b \in W$ is given.
	Such a situation appears, for example, in convection-diffusion equations with a prescribed velocity field (see \sref{sec: examples} below).
	
	(ii) Under additional assumptions on $B$ and $\varphi$ one can obtain global-in-time existence of a strong solution, which corresponds either to the 2D Navier--Stokes equations with the cancelation property or to the 3D Navier--Stokes equations with small data.
	We will briefly mention these results at Remarks \ref{rem: large data global existence} and \ref{rem: small data global existence} below.
\end{rem}

\section{Examples} \label{sec: examples}
We discuss applications of Theorems \ref{main thm1} and \ref{main thm2} to concrete problems, checking the validity of (H5).
\subsection{Obstacle problem for convection-diffusion equations} \label{subsec: obstacle}
Let $\Omega \subset \mathbb R^d \, (d = 1, 2, 3)$ be a bounded domain with the sufficiently smooth boundary, $H := L^2(\Omega)$, $V := H^1(\Omega)$, $W : = H^2(\Omega)$, $a(u, v) := (\nabla u, \nabla v)$, $B(u, v) := b(u) \cdot \nabla v$, where $b : \mathbb R \to \mathbb R^d$ is a linear function, and
\begin{equation*}
	\varphi(v) := \begin{cases}
		0 & \quad \text{if} \quad v \in K:= \{v \in H^1_0(\Omega) \mid v \ge 0 \; \text{a.e.\ in} \; \Omega \}, \\
		+ \infty & \quad \text{if} \quad v \in V \setminus K.
	\end{cases}
\end{equation*}
We know from \cite[Proposition 2.11]{Bar2010} and \cite[Corollary 14]{Bre1971} that hypothesis (H5) is satisfied.
Therefore, by Theorems \ref{main thm1} and \ref{main thm2}, there exists a unique local-in-time strong solution $u$, which has $H^2(\Omega)$-spatial regularity, such that
\begin{alignat*}{2}
	\partial_t u - \Delta u + b(u) \cdot \nabla u &= f \qquad && \text{in} \quad \{(x, t) \in \Omega \times (0, T_*) \mid u(x, t) > 0\}, \\
	\partial_t u - \Delta u + b(u) \cdot \nabla u &\ge f, \quad u \ge 0 \qquad && \text{in} \quad \Omega \times (0, T_*), \\
	u &= 0 \qquad && \text{on} \quad \partial\Omega \times (0, T_*), \\
	u &= u^0 \qquad && \text{on} \quad \Omega \times \{0\},
\end{alignat*}
depending on regularity assumptions of $u^0$ and $f$.
We remark that, although this was already obtained in \cite[Corollary 5.2]{Bar2010} for $b \equiv 0$, extension to $b \neq 0$ was non-trivial---even if $b(u) \equiv b$ (given smooth velocity field; recall \rref{rem: after main thm}(i))---because the monotonicity $\left< B(\cdot, u - v), u - v \right> \ge 0$ does not necessarily hold (note that neither $\operatorname{div} b = 0$ in $\Omega$ nor $b \cdot \nu \ge 0$ on $\partial\Omega$ is valid for a general vector field $b$).

\subsection{Nonlinear Neumann problems for convection-diffusion equations}
We adopt the same setting as in Subsection \ref{subsec: obstacle} except that we replace $\varphi : V \to (-\infty, +\infty]$ by the following:
\begin{equation*}
	\varphi(v) := \begin{cases}
		\int_{\partial\Omega} j(v) \, d\gamma & \quad \text{if} \quad v \in H^1(\Omega), \; j(v) \in L^1(\partial\Omega), \\
		+ \infty & \quad \text{otherwise},
	\end{cases}
\end{equation*}
where $j : \mathbb R \to (-\infty, +\infty]$ is a convex, proper, and lower-semicontinuous function, and $d\gamma$ denotes the surface measure of $\partial\Omega$.
We know from \cite[Proposition 2.9]{Bar2010}) and \cite[Theorem 12]{Bre1971} that hypothesis (H5) is satisfied.
Therefore, by Theorems \ref{main thm1} and \ref{main thm2}, under suitable assumptions on $u^0$ and $f$, we obtain a local-in-time strong solution $u$ such that
\begin{alignat*}{2}
	\partial_t u - \Delta u + b(u) \cdot \nabla u &= f \qquad && \text{in} \quad \Omega \times (0, T_*), \\
	-\frac{\partial u}{\partial\nu} &\in \beta(u) \qquad && \text{on} \quad \partial\Omega \times (0, T_*), \\
	u &= u^0 \qquad && \text{on} \quad \Omega \times \{0\},
\end{alignat*}
where $\beta := \partial j$ is a maximal monotone graph of $\mathbb R \times \mathbb R$ and recall that $\nu$ is the outer unit normal to $\partial\Omega$.
In particular, the choices $j(r) = + \infty \, (r < 0)$, $j(r) = 0 \, (r \ge 0)$ and $j(r) = |r|$ correspond to simplified Signorini and friction problems, respectively (cf.\ \cite[Sections II.4 and II.5]{Glo2008}).

\subsection{Navier--Stokes equations with boundary conditions of friction type}
Let $\Omega \subset \mathbb R^d \, (d = 2, 3)$ be a bounded smooth domain, in which $\partial\Omega$ has two components $\Gamma_D$ and $\Gamma$ (i.e., $\partial\Omega = \Gamma_D \cup \Gamma$ and $\Gamma_D \cap \Gamma = \emptyset$).
Furthermore, we set $H := \{ v \in L^2(\Omega)^d \mid \operatorname{div} v = 0 \text{ in } \Omega, \; v \cdot \nu = 0 \text{ on } \Gamma_D \}$, $V := \{ v \in H^1(\Omega)^d \cap H \mid v = 0 \text{ on } \Gamma_D \}$, $W : = H^2(\Omega)^d \cap V$, $a(u, v) := 2(\mathbb E(u), \mathbb E(v))$, where $\mathbb E(u) = \frac{\nabla u + (\nabla u)^\top}{2}$ denotes the rate-of-strain tensor of $u$, and $B(u, v) := u \cdot \nabla v$.
Finally, we consider $\varphi_\tau : V \to (-\infty, +\infty]$ (resp.\ $\varphi_\nu : V \to (-\infty, +\infty]$) defined by
\begin{align*}
	&\varphi_\tau(v) := \begin{cases}
		\int_\Gamma g|v_\tau| \, d\gamma & \quad \text{if} \quad v \in V_\tau := \{v \in V \mid v \cdot \nu = 0 \; \text{on} \; \Gamma \}, \\
		+ \infty & \quad \text{if} \quad v \in V \setminus V_\tau,
	\end{cases} \\
	\bigg( \text{resp.} \quad &\varphi_\nu(v) := \begin{cases}
		\int_\Gamma g|v \cdot \nu| \, d\gamma & \quad \text{if} \quad v \in V_\nu := \{v \in V \mid v_\tau = 0 \; \text{on} \; \Gamma \}, \\
		+ \infty & \quad \text{if} \quad v \in V \setminus V_\nu,
	\end{cases} \bigg)
\end{align*}
where $g \in H^1(\Gamma)$, $g \ge 0$ is a given function and we recall $v_\tau = v - (v \cdot \nu) \nu$.
We know from \cite{Sai2004} that hypothesis (H5) is satisfied in the above setting.
Therefore, by Theorems \ref{main thm1} and \ref{main thm2}, under suitable assumptions on $u^0$ and $f$, we obtain a local-in-time strong solution $u$ admitting $H^2(\Omega)$-spatial regularity of (the pressure is represented by $p$):
\begin{alignat*}{2}
	\partial_t u - \Delta u + u \cdot \nabla u + \nabla p &= f, \quad \operatorname{div} u = 0 \qquad && \text{in} \quad \Omega \times (0, T_*), \\
	u& = 0 \qquad && \text{on} \quad \Gamma_D \times (0, T_*), \\
	u \cdot \nu = 0, \quad -\sigma_\tau(u) &\in g \partial |u_\tau| \qquad && \text{on} \quad \Gamma \times (0, T_*), \\
	\big( \text{resp.} \quad u \cdot \tau = 0, \quad -\sigma_\nu(u, p) &\in g \partial |u \cdot \nu| \qquad && \text{on} \quad \Gamma \times (0, T_*), \big) \\
	u &= u^0 \qquad && \text{on} \quad \Omega \times \{0\},
\end{alignat*}
where we see that the slip (resp.\ leak) boundary condition of friction type is imposed on $\Gamma$.
Here, $\sigma_\tau(u)$ and $\sigma_\nu(u, p)$ are the tangential and normal components of the traction vector $\sigma(u, p) \nu = -p \nu + 2 \mathbb E(u)\nu$, respectively.
We remark that time-dependent $g$'s are treated by \cite{Kas2013JDE}; however, $H^2(\Omega)$-regularity cannot be shown even for time-independent $g$'s by the method of \cite{Kas2013JDE} (see p.\ 777 there).
We also note that $(u\cdot \nabla v, v) = 0$ holds for the slip case, thus admitting a global-in-time strong solution if $d = 2$.

\subsection{Bingham problem with perfect slip boundary condition}
Let $\Omega \subset \mathbb R^d \, (d = 2, 3)$ be a bounded smooth domain, $H := \{ v \in L^2(\Omega)^d \mid \operatorname{div} v = 0 \text{ in } \Omega, \; v \cdot \nu = 0 \text{ on } \partial\Omega \}$, $V := H^1(\Omega)^d \cap H$, $W : = H^2(\Omega)^d \cap H$, $a(u, v) := (\nabla u, \nabla v)$, $B(u, v) := u \cdot \nabla v$, and $\varphi(v) := \int_\Omega g |\nabla v| \, dx$, where $g \in H^1_0(\Omega)$, $g \ge 0$ is a given function.
We know from \cite[Corollary 3.1]{FuKa2026} that hypothesis (H5) is satisfied.
Therefore, by Theorems \ref{main thm1} and \ref{main thm2} we obtain a local-in-time strong solution of
\begin{alignat*}{2}
	\partial_t u + u \cdot \nabla u + \operatorname{div}(T^D - p I) &= f, \quad \operatorname{div} u = 0 \qquad && \text{in} \quad \Omega \times (0, T_*), \\
	&\hspace{-3cm} \begin{cases}
		T^D = \nabla u + g \frac{\nabla u}{|\nabla u|} \quad & \text{if} \quad \nabla u \neq 0 \\
		|T^D| \le g  \quad & \text{if} \quad \nabla u = 0
	\end{cases} \qquad && \text{in} \quad \Omega \times (0, T_*), \\
	u \cdot \nu = 0, \quad (\nabla u \, \nu)_\tau &= 0 \qquad && \text{on} \quad \partial\Omega \times (0, T_*), \\
	u &= u^0 \qquad && \text{on} \quad \Omega \times \{0\},
\end{alignat*}
where $T^D$ and $p$ constitute the deviatoric and pressure parts of of the stress tensor, respectively.
We remark that strong solvability with $H^2(\Omega)$-regularity for time-dependent $g$'s is shown by \cite[Theorems 4.1 and 4.2]{FuKa2026}; hence the above result, in which $g$ is assumed not to depend on $t$, is weaker than \cite{FuKa2026}.

\section{Solvability of stationary problems} \label{sec3}
Sections \ref{sec3}--\ref{sec5} will be devoted to the proof of \tref{main thm1}.
Let us first recall a well-known result of elliptic variational inequalities (see e.g.\ \cite[Theorem I.4.1]{Glo2008} and \cite[Corollary 2.9]{Bar2010}): for all $\lambda \ge 0$ and $f \in V'$, there exists a unique solution $u \in V \cap D(\partial\varphi)$ such that
\begin{equation} \label{eq: stationary Stokes VI}
	\lambda (u, v - u) + a(u, v - u) + \varphi(v) - \varphi(u) \ge \left< f, v - u \right> \qquad \forall v \in V.
\end{equation}
Therefore, $\lambda I + A + \partial\varphi : V \to V'$ is bijective.
Since $A + \partial\varphi$ is monotone by the coercivity of $a$ and \cite[Theorem 2.8]{Bar2010}, we find from \cite[Theorem 2.2]{Bar2010} that $A + \partial\varphi : V \to V'$ is maximal monotone.
Its restriction to $H$, i.e., $(A + \partial\varphi)_H : H \to 2^H; \; u \mapsto (A + \partial\varphi)(u) \cap H$, with $D((A + \partial\varphi)_H) := \{ u \in H \mid Au \in H, \, \partial\varphi(u) \cap H \neq \emptyset \}$, is also a maximal monotone operator.
From now on, we will not distinguish $(A + \partial\varphi)_H$ from $A + \partial\varphi$.

In the next lemma, we construct a solution of an Oseen version of \eref{eq: stationary Stokes VI} for sufficiently large $\lambda > 0$.
\begin{lem} \label{lem: existence for Oseen VI}
	Let $w \in V$ and $f \in H$.
	For all $\lambda \ge C_{\theta_1, 1/4} \|w\|_V^{\theta_1}$, there exists some $u \in V $ such that
	\begin{equation} \label{eq: stationary Oseen VI}
		\lambda (u, v - u) + a(u, v - u) + \left< B(w, u), v - u \right> + \varphi(v) - \varphi(u) \ge (f, v - u) \qquad \forall v \in V.
	\end{equation}
	In particular, we have $u \in D(\partial\varphi)$.
\end{lem}
\begin{proof}
	Let $L = L_{\lambda, w} : V \to V'$ be defined by $L(u) = \lambda u + Au + B(w, u)$, which is clearly a bounded (linear) operator.
	We observe that $L$ is pseudo-monotone; in fact, for a sequence $u_j \rightharpoonup u$ weakly in $V$ as $j \to \infty$, one has
	\begin{align*}
		\liminf_{j \to \infty} \left< L(u_j), u_j - v \right> &= \lim_{j \to \infty} \left< \lambda u_j + B(w, u_j), u_j - v \right> + \liminf_{j \to \infty} a(u_j, u_j - v) \\
			&\ge \left< \lambda u + B(w, u), u - v \right> + a(u, u - v) = \left< L(u), u - v \right>,
	\end{align*}
	where we have used compactness of $V \hookrightarrow H$, hypothesis (H2), and lower semi-continuity of the norm $\sqrt{a(\cdot, \cdot)} = \|\cdot\|_V$.
	Moreover, $L$ is coercive, that is, for some $v_0 \in D(\varphi)$
	\begin{equation*}
		\frac{\left< L(v), v - v_0 \right> + \varphi(v)}{\|v\|_V} \to + \infty \quad\text{as}\quad \|v\|_V \to \infty.
	\end{equation*}
	In fact, for any $\|v\|_V \ge 1$ and $v_0 \in D(\varphi) \neq \emptyset$ we have (recall $\|v\|_V^2 = a(v, v)$)
	\begin{align*}
		\left< L(v), v - v_0 \right> + \varphi(v) &= \lambda \|v\|_H^2 + \|v\|_V^2 + \left< B(w, v), v \right> + \varphi(v) - \left< \lambda v + Av + B(w, v), v_0 \right> \\
			&\ge (\lambda - C_{\theta_1, 1/4} \|w\|_V^{\theta_1}) \|v\|_H^2 + \frac{3}{4} \|v\|_V^2 - C_{\varphi 1} (\|v\|_V + 1) - (\lambda + 1 + C_B \|w\|_V) \|v\|_V \|v_0\|_V \\
			&\ge \frac{3}{4} \|v\|_V^2 - \big( 2C_{\varphi 1} + (\lambda + 1 + C_B \|w\|_V) \|v_0\|_V \big) \|v\|_V,
	\end{align*}
	where we have used hypothesis (H1), \eref{eq: phi is bounded from below}, and \eref{eq: boundedness of B}.
	Then an abstract surjectivity result for pseudo-monotone operators (see \cite[Th\'{e}or\`{e}m II.8.5]{Lio1969}) asserts the existence of a solution of \eref{eq: stationary Oseen VI}.
\end{proof}

The solution constructed in \lref{lem: existence for Oseen VI} indeed admits more regularity $u \in W$ under hypothesis (H5).
\begin{lem} \label{lem: regularity for Oseen VI}
	In the setting of \lref{lem: existence for Oseen VI}, if $\lambda \ge C_{\theta_1, 1/4} (2\|w\|_V)^{\theta_1} + 1/2$, any solution $u$ of \eref{eq: stationary Oseen VI} belongs to $W$.
\end{lem}
\begin{proof}
	If $w \in W$, we immediately have $u \in W$ from (H5) by arranging terms of \eref{eq: stationary Oseen VI} as follows (note that $B(w, u) \in H$ by (H3)):
	\begin{equation*}
		a(u, v - u) + \varphi(v) - \varphi(u) \ge (f - \lambda u - B(w, u), v - u) \qquad \forall v \in V.
	\end{equation*}

	Now consider a general $w \in V$.
	Since $W \hookrightarrow V$ is dense, there exist $w_j \in W \, (j= 1, 2, \dots)$ such that $w_j \to w$ strongly in $V$ as $j \to \infty$.
	Without loss of generality we may assume $\|w_j\|_V \le 2 \|w\|_V$.
	By the above argument, one can find $u_j \in W$ such that
	\begin{equation} \label{eq: Oseen VI approximated by wj}
		\lambda (u_j, v - u_j) + a(u_j, v - u_j) + (B(w_j, u_j), v - u_j) + \varphi(v) - \varphi(u_j) \ge (f, v - u_j) \qquad \forall v \in V.
	\end{equation}
	
	We take $v$ to be any $v_0 \in D(\varphi)$ and obtain
	\begin{align*}
		\lambda \|u_j\|_H^2 + \|u_j\|_V^2 &\le (f - B(w_j, u_j), u_j) - \varphi(u_j) + \left< \lambda u_j + A u_j - B(w_j, u_j) - f, v_0 \right> + \varphi(v_0) \\
			&\le \|f\|_H^2 + \frac12 \|u_j\|_H^2 + C_{\theta_1, 1/4} \|w_j\|_V^{\theta_1} \|u_j\|_H^2 + \frac{1}{4} \|u_j\|_V^2 + C_{\varphi 1} (\|u_j\|_V + 1) + C_{\lambda, C_B, \|w\|_V} \|v_0\|_V^2 + \varphi(v_0),
	\end{align*}
	where $C_{\lambda, C_B, \|w\|_V} > 0$ is a constant independent of $j$.
	Therefore,
	\begin{equation*}
		\frac{1}{2} \|u_j\|_V^2 \le \|f\|_H^2 + C_{\varphi 1}^2 + C_{\varphi 1} + C_{\lambda, A, B, w} \|v_0\|_V^2 + \varphi(v_0),
	\end{equation*}
	which gives a uniform bound of $u_j$ in $V$ (and $H$ as well).
	We further find from \eref{eq: Oseen VI approximated by wj}, (H5), and (H4) that
	\begin{align*}
		\|u_j\|_W &\le C_{\mathrm{reg}} (\|f\|_H + \lambda \|u_j\|_H + \|B(w_j, u_j)\|_H) + C_{\varphi 2} \\
			&\le C_{\mathrm{reg}} (\|f\|_H + \lambda \|u_j\|_H) + C_{\mathrm{reg}}C_{\theta_2, \frac{1}{2 C_{\mathrm{reg}}}} \|w_j\|_V^{\theta_2} \|u_j\|_V + \frac12 \|u_j\|_W  + C_{\varphi 2},
	\end{align*}
	which shows that $\{u_j\}$ is uniformly bounded in $W$.
	Therefore, there exists a subsequence $\{u_{j'}\}$ and $u_* \in W$ such that $u_{j'} \rightharpoonup u_*$ weakly in $W$.
	Since $W \hookrightarrow V$ is compact, $u_{j'} \to u_*$ strongly in $V$.
	
	Finally we claim that $u = u_*$, which proves the lemma.
	For this it suffices to show $u_j \to u$ strongly in $V$ as $j \to \infty$.
	We substitute $v = u$ and $v = u_j$ in \eref{eq: Oseen VI approximated by wj} and \eref{eq: stationary Oseen VI} respectively, and add the resulting inequalities to obtain
	\begin{equation} \label{eq: similar to uniqueness of un}
	\begin{aligned}
		\lambda \|u - u_j\|_H^2 + \|u - u_j\|_V^2 &\le \left< B(w_j, u_j) - B(w, u), u - u_j \right> = \left< B(w_j, u_j - u), u - u_j \right> + \left< B(w_j - w, u), u - u_j \right> \\
			&\le C_{\theta_1, 1/4} \|w_j\|_V^{\theta_1} \|u - u_j\|_H^2 + \frac{1}{4} \|u - u_j\|_V^2 + C_B \|w_j - w\|_V \|u\|_V \|u - u_j\|_V,
	\end{aligned}
	\end{equation}
	where we have used (H1).
	Absorbing the first and second terms on the right-hand side yields
	\begin{equation*}
		\frac{3}{4} \|u - u_j\|_V \le C_B \|w_j - w\|_V \|u\|_V \to 0 \quad\text{as}\quad j \to \infty,
	\end{equation*}
	which completes the proof.
\end{proof}
\begin{rem}
	A similar calculation to \eref{eq: similar to uniqueness of un} shows that the solution of \eref{eq: stationary Oseen VI} is unique.
\end{rem}

\section{Discrete-in-time approximation and a priori estimates} \label{sec4}
We propose the following discrete-in-time approximation to \eref{eq: VI}, which is semi-implicit with respect to the Navier--Stokes type nonlinearity $B$ (and fully implicit with respect to the monotone type nonlinearity $\varphi$).
Given $u^0 \in D(\varphi) \subset V$, $f \in L^2(0, T; H)$, and $\Delta t > 0$, we try to find $\{u^n\}_{n=1}^N \subset D(\varphi)$ such that, for $n = 1, \dots, N$,
\begin{equation} \label{eq: semi-implicit scheme}
	\Big( \frac{u^n - u^{n-1}}{\Delta t}, v - u^n \Big) + a(u^n, v - u^n) + \left< B(u^{n-1}, u^n), v - u^n \right> + \varphi(v) - \varphi(u^n) \ge (f^n, v - u^n) \qquad \forall v \in V,
\end{equation}
where $f^n := \frac{1}{\Delta t} \int_{(n-1)\Delta t}^{n \Delta t} f(t) \, dt$.
According to Lemmas \ref{lem: existence for Oseen VI} and \ref{lem: regularity for Oseen VI}, the above variational inequalities indeed admit (unique) solutions $u^n \in W \cap D(\partial\varphi) \, (j = 1, \dots, N)$, provided that $N$ satisfies
\begin{equation*}
	\frac{1}{\Delta t} \ge C_{\theta_1, 1/4} (2 \max_{1\le n\le N} \|u^{n-1}\|_V)^{\theta_1} + \frac12.
\end{equation*}
We need to ensure that such $N$, which depends on $\Delta t$, can be taken large enough for $N \Delta t$ to admit a lower bound $T_* > 0$ that is independent of $\Delta t$.
For this purpose, we establish an \emph{a priori} estimate as follows.

\begin{prop}
	Suppose that $\{u^n\}_{n=1}^N \subset W$ solve \eref{eq: semi-implicit scheme}.
	Then, for all $n = 1, \dots, N$ we have
	\begin{equation} \label{eq: H1 a priori estimate}
		\Big\| \frac{u^n - u^{n-1}}{\Delta t} \Big\|_H^2 + \frac{\|u^n\|_V^2 - \|u^{n-1}\|_V^2 + \|u^n - u^{n-1}\|_V^2}{\Delta t}
			+ \frac{2( \varphi(u^n) - \varphi(u^{n-1}) )}{\Delta t}
			\le M (\|u^{n-1}\|_V^{2 \theta_2} \|u^n\|_V^2 + \|f^n\|_H^2 + 1),
	\end{equation}
	where $M > 0$ is a constant dependent only on $\theta_2$, $C_{\mathrm{reg}}$, and $C_{\varphi 2}$.
\end{prop}
\begin{proof}
	Application of (H5) to \eref{eq: semi-implicit scheme}, combined with (H4), leads to
	\begin{align*}
		\|u^n\|_W &\le C_{\mathrm{reg}} (\|f^n\|_H + \Big\| \frac{u^n - u^{n-1}}{\Delta t} \Big\|_H + \|B(u^{n-1}, u^n)\|_H) + C_{\varphi 2} \\
			&\le C_{\mathrm{reg}} (\|f^n\|_H + \Big\| \frac{u^n - u^{n-1}}{\Delta t} \Big\|_H + C_{\theta_2, \frac{1}{2 C_{\mathrm{reg}}}} \|u^{n-1}\|_V^{\theta_2} \|u^n\|_V) + \frac12 \|u^n\|_W + C_{\varphi 2},
	\end{align*}
	whence, by $(a + b + c + d)^2 \le 4(a^2 + b^2 + c^2 + d^2)$,
	\begin{equation} \label{eq2: proof of a priori estimate}
		\|u^n\|_W^2 \le 16 C_{\mathrm{reg}}^2 (\|f^n\|_H^2 + \Big\| \frac{u^n - u^{n-1}}{\Delta t} \Big\|_H^2 + C_{\theta_2, \frac{1}{2 C_{\mathrm{reg}}}}^2 \|u^{n-1}\|_V^{2 \theta_2} \|u^n\|_V^2) + 16C_{\varphi 2}^2.
	\end{equation}
	
	Next we take $v = u^{n-1}$ in \eref{eq: semi-implicit scheme} and divide by $\Delta t$ to obtain
	\begin{align*}
		\Big\| \frac{u^n - u^{n-1}}{\Delta t} \Big\|_H^2 + \frac{\|u^n\|_V^2 - \|u^{n-1}\|_V^2 + \|u^n - u^{n-1}\|_V^2}{2\Delta t} + \frac{\varphi(u^n) - \varphi(u^{n-1})}{\Delta t}
			\le \Big( f^n - B(u^{n-1}, u^n), \frac{u^n - u^{n-1}}{\Delta t} \Big),
	\end{align*}
	where we recall $a(v, v) = \|v\|_V^2$.
	Since the right-hand side is bounded by
	\begin{equation*}
		\Big\| f^n - B(u^{n-1}, u^n) \Big\|_H^2 + \frac14 \Big\| \frac{u^n - u^{n-1}}{\Delta t} \Big\|_H^2
			\le 2 \|f^n\|_H^2 + 2 \|B(u^{n-1}, u^n)\|_H^2 + \frac14 \Big\| \frac{u^n - u^{n-1}}{\Delta t} \Big\|_H^2,
	\end{equation*}
	it follows from (H5) that
	\begin{equation} \label{eq1: proof of a priori estimate}
	\begin{aligned}
		&\frac34 \Big\| \frac{u^n - u^{n-1}}{\Delta t} \Big\|_H^2 + \frac{\|u^n\|_V^2 - \|u^{n-1}\|_V^2 + \|u^n - u^{n-1}\|_V^2}{2\Delta t} + \frac{\varphi(u^n) - \varphi(u^{n-1})}{\Delta t} \\
		\le \; &2 \|f^n\|_H^2 + 4 ( C_{\theta_2, \epsilon}^2 \|u^{n-1}\|_V^{2 \theta_2} \|u^n\|_V^2 + \epsilon^2 \|u^n\|_W^2).
	\end{aligned}
	\end{equation}
	Taking $\epsilon$ such that $4\epsilon^2 \cdot 16C_{\mathrm{reg}}^2 = 1/4$ (thus $\epsilon = 1/(16C_{\mathrm{reg}})$) and substituting \eref{eq1: proof of a priori estimate} into \eref{eq2: proof of a priori estimate}, we deduce that
	\begin{align*}
		&\frac12 \Big\| \frac{u^n - u^{n-1}}{\Delta t} \Big\|_H^2 + \frac{\|u^n\|_V^2 - \|u^{n-1}\|_V^2 + \|u^n - u^{n-1}\|_V^2}{2\Delta t} + \frac{\varphi(u^n) - \varphi(u^{n-1})}{\Delta t} \\
		\le \; &\frac94 \|f^n\|_H^2
			+ \big(4 C_{\theta_2, \frac{1}{16C_{\mathrm{reg}}}}^2 + C_{\theta_2, \frac{1}{2C_{\mathrm{reg}}}}^2/4 \big) \|u^{n-1}\|_V^{2 \theta_2} \|u^n\|_V^2
			+ \frac{C_{\varphi 2}^2}{4 C_{\mathrm{reg}}^2},
	\end{align*}
	which is the desired inequality.
\end{proof}

Setting $b_n := \|u^n\|_V^2 + 2 \varphi(u^n)$, we have
\begin{equation*}
	\|u^n\|_V^2 \le b_n + 2 C_{\varphi 1}(\|u^n\|_V + 1) \le b_n + 2 C_{\varphi 1}^2 + \frac12 \|u^n\|_V^2 + 2 C_{\varphi 1},
\end{equation*}
so that
\begin{equation*}
	\|u^n\|_V^2 \le 2 b_n + \underbrace{4 C_{\varphi 1}(C_{\varphi 1} + 1)}_{ =: C_{\varphi 3} } \Longleftrightarrow b_n \ge \frac12 \|u^n\|_V^2 - \frac12 C_{\varphi 3}.
\end{equation*}
Substitution of this into \eref{eq: H1 a priori estimate} yields
\begin{align*}
	&\frac{b_n - b_{n-1}}{\Delta t} \le M (2b_{n-1} + C_{\varphi 3})^{\theta_2} (2b_n + C_{\varphi 3}) + M(\|f^n\|_H^2 + 1) \\
	\Longleftrightarrow \; & \frac{c_n - c_{n-1}}{\Delta t} \le 2^{\theta_2 + 1} M c_{n-1}^{\theta_2} c_n + M(\|f^n\|_H^2 + 1),
\end{align*}
where $c_n := b_n + C_{\varphi 3}/2$ is non-negative.
We further introduce $d_n := c_n + \beta$, in which $\beta > 0$ is a constant to be fixed later (see \eref{eq: condition on beta} below).
Then, setting $M' := 2^{\theta_2 + 1} M$, we see that
\begin{equation} \label{eq: tilde bn}
	\frac{d_n - d_{n-1}}{\Delta t} \le M'(d_{n-1}^{\theta_2} d_n + \|f^n\|_H^2 + 1).
\end{equation}
Let us solve this difference inequality as follows.

\begin{lem} \label{lem: difference inequality}
	Let $\{x_n\}$ and $\{y_n\}$ be sequences of positive numbers such that $x_n \ge \beta > 0 \, (n = 0, 1, \dots)$.
	Assume that they satisfy
	\begin{equation*}
		\frac{x_n - x_{n-1}}{\Delta t} \le M(x_{n-1}^\theta x_n + y_n) \quad n = 1, 2, \dots,
	\end{equation*}
	where $\Delta t > 0$, $M > 0$, and $\theta \ge 1$ are constants.
	Then we have $x_n \le 2^{1/\theta} x_0$ for all $n$ that satisfies
	\begin{equation} \label{eq: range of good n}
		4M x_0^\theta (n \Delta t) \le \theta^{-1} \quad\text{and}\quad 4M x_0^\theta \sum_{m=1}^n y_m \Delta t \le \beta^{\theta+1}.
	\end{equation}
\end{lem}
\begin{proof}
	Suppose $x_n \ge x_{n-1}$.
	Then, since $1 - (\frac{x_{n-1}}{x_n})^\theta \le \theta (1 - \frac{x_{n-1}}{x_n}) \Leftrightarrow x_n^\theta - x_{n-1}^\theta \le \theta x_n^{\theta-1} (x_n - x_{n-1})$, it follows that
	\begin{equation*}
		x_{n-1}^{-\theta} - x_n^{-\theta} = \frac{x_n^\theta - x_{n-1}^\theta}{x_{n-1}^\theta x_n^\theta} \le \frac{\theta(x_n - x_{n-1})}{x_{n-1}^\theta x_n} \le M \Delta t \Big( \theta + \frac{y_n}{\beta^{\theta + 1}} \Big).
	\end{equation*}
	This obviously holds for the other case $x_n \le x_{n-1}$ as well.
	Consequently, by summation we obtain
	\begin{equation*}
		x_0^{-\theta} - x_n^{-\theta} \le M \theta t_n + M\beta^{-\theta-1} \sum_{m=1}^n y_m \Delta t
		\Longleftrightarrow x_n \le \Big( x_0^{-\theta} - M \theta t_n - M\beta^{-\theta-1} \sum_{m=1}^n y_m \Delta t) \Big)^{-1/\theta},
	\end{equation*}
	where $t_n := n\Delta t$.
	We conclude $x_n \le 2^{1/\theta} x_0$ in the range of $n$ where \eref{eq: range of good n} holds.
\end{proof}
\begin{rem}
	The ``fully-implicit'' version of the difference inequality above, i.e.,
	\begin{equation*}
		\frac{x_n - x_{n-1}}{\Delta t} \le M(x_n^{\theta+1} + y_n)
	\end{equation*}
	cannot be solved easily.
	For this reason, the adoption of the semi-implicit scheme \eref{eq: semi-implicit scheme} is essential for our argument.
\end{rem}
\begin{rem} \label{rem: large data global existence}
	Consider a special case where the cancelation property $\left< B(u, v), v \right> = 0$ is valid and an estimate
	\begin{equation*}
		\|B(u, v)\|_H \le C \|u\|_H^{1/2} \|u\|_V^{1/2} \|v\|_V^{1/2} \|v\|_W^{1/2} \qquad (u \in V, \, v \in W),
	\end{equation*}
	which corresponds to the 2D Navier--Stokes equations, is available.
	Then we can solve \eref{eq: semi-implicit scheme} without any restrictions on $\Delta t$ and $N$.
	Moreover, taking $v = 0$ in \eref{eq: semi-implicit scheme}, we have a uniform bound for $\|u^n\|_H^2 + \sum_{m=1}^n \|u^m\|_V^2 \Delta t$, provided that $0 \in D(\varphi)$.
	We arrive at a difference inequality similar to \eref{eq: tilde bn} where $d_{n-1}^{\theta_2} d_n$ is replaced by $\|u^{n-1}\|_H^2 d_{n-1} d_n$, to which a discrete Gronwall inequality---instead of \lref{lem: difference inequality}---is applicable.
	This will lead to a global-in-time existence of a solution of \eref{eq: VI}.
\end{rem}

As a result of this lemma applied to $d_n \ge \|u^n\|_V^2/2$, we obtain (note that $\sum_{m=1}^n \|f^m\|_H^2 \Delta t \le \|f\|_{L^2(0, n \Delta t; H)}^2$):
\begin{cor} \label{cor: a priori estimate of un and T*}
	Suppose $\{u^n\}_{n=1}^N \subset W$ solve \eref{eq: semi-implicit scheme} and let $\beta > 0$ be large enough to satisfy
	\begin{equation} \label{eq: condition on beta}
		4M' (\|u^0\|_V^2 + 2\varphi(u^0) + C_{\varphi 3}/2 + \beta)^{\theta_2} \|f\|_{L^2(0, T; H)}^2 \le \beta^{\theta_2 + 1},
	\end{equation}
	where $M'$ is as in \eref{eq: tilde bn} (such $\beta$ indeed exists since $\theta_2 \ge 0$).
	Then, for all $N \in \mathbb N$ such that $N \Delta t \le T$ and
	\begin{equation} \label{eq: condition on N}
		4M' (\|u^0\|_V^2 + 2\varphi(u^0) + C_{\varphi 3}/2 + \beta)^{\theta_2} (N \Delta t) \le \theta_2^{-1},
	\end{equation}
	$\{u^n\}_{n=1}^N$ admits the following upper bound in $V$:
	\begin{equation} \label{eq: final H1 upper bound}
		\|u^n\|_V^2 \le 2d_n \le 2^{1 + 1/\theta_2} (\|u^0\|_V^2 + 2\varphi(u^0) + C_{\varphi 3}/2 + \beta) \qquad (n = 1, \dots, N).
	\end{equation}
\end{cor}

We are now ready to state a complete existence result for the discrete-in-time approximation \eref{eq: semi-implicit scheme}.
\begin{thm} \label{thm: well-posedness and estimate of semi-implicit scheme}
	Let $u_0 \in V \cap D(\varphi)$ and $f \in L^2(0, T; H)$, and choose $\beta$ that satisfies \eref{eq: condition on beta}.
	For all $\Delta t > 0$ and $N \in \mathbb N$ such that
	\begin{equation*}
		\frac{1}{\Delta t} \ge C_{\theta_1, 1/4} 2^{\theta_1} \Big( 2^{1 + 1/\theta_2} (\|u^0\|_V^2 + 2\varphi(u^0) + \frac{C_{\varphi 3}}{2} + \beta) \Big)^{\theta_1} + \frac12
	\end{equation*}
	and \eref{eq: condition on N} hold, there exist a unique solution $\{u^n\}_{n=1}^N \subset W \cap D(\partial\varphi)$ of \eref{eq: semi-implicit scheme}.
	Moreover, it satisfies
	\begin{equation*}
		\sum_{m=1}^n \Big( \Big\| \frac{u^m - u^{m-1}}{\Delta t} \Big\|_H^2 + \|u^n\|_W^2 \Big) \Delta t + \|u^n\|_V^2 + \sum_{m=1}^n \|u^m - u^{m-1}\|_V^2 \le K \qquad (1\le n\le N),
	\end{equation*}
	where $K > 0$ is a constant depending only on $u^0, f, \varphi, \theta_1, \theta_2, C_{\mathrm{reg}}$, and independent of $\Delta t$, $N$, and $n$.
\end{thm}
\begin{proof}
	Unique existence of a solution of \eref{eq: semi-implicit scheme} follows from \lref{lem: existence for Oseen VI}, \lref{lem: regularity for Oseen VI}, and \cref{cor: a priori estimate of un and T*} by induction on $n$.
	Multiplying \eref{eq: H1 a priori estimate} by $\Delta t$ and adding it for $m = 1, \dots, n$, we obtain
	\begin{align*}
		&\sum_{m=1}^n \Big\| \frac{u^m - u^{m-1}}{\Delta t} \Big\|_H^2 \Delta t + \|u^n\|_V^2 + 2 \varphi(u^n) + \sum_{m=1}^n \|u^m - u^{m-1}\|_V^2 \\
		\le \; & \|u^0\|_V^2 + 2 \varphi(u^0) + M \Big[ (\max_{0 \le n \le N} \|u^n\|_V^2)^{\theta_2 + 1} + 1 \Big] N\Delta t + M \|f\|_{L^2(0, T; H)}^2 \quad (n = 1, \dots, N).
	\end{align*}
	We conclude the desired estimate from $-2 \varphi(u^n) \le 2C_{\varphi 1}(C_{\varphi 1} + 1) + \|u^n\|_V^2/2$, \eref{eq: final H1 upper bound}, and \eref{eq: condition on N}.
\end{proof}

\section{Passage to limit $\Delta t \to 0$ and proof of \tref{main thm1}} \label{sec5}
Let us define an ``existence time interval length'' $T_*$ by
\begin{equation*}
	T_* := \min\{ \big[ 8 M' (\|u^0\|_V^2 + 2\varphi(u^0) + C_{\varphi 3}/2 + \beta)^{\theta_2} \theta_2 \big]^{-1}, T \} > 0.
\end{equation*}
\begin{rem} \label{rem: small data global existence}
	Consider a special case where $\varphi : V \to \mathbb R$ is continuous, non-negative, and $\varphi(0) = 0$ (this means $C_{\varphi 1} = C_{\varphi 3} = 0$ by \eref{eq: phi is bounded from below}).
	Assume in addition that the data $\|u^0\|_V$ and $\|f\|_{L^2(0, T; H)}$ are sufficiently small.
	Then we can find a suitable $\beta > 0$ such that \eref{eq: condition on beta} holds and the $T_*$ above equals an arbitrarily fixed $T$.
	This will lead to a global existence result for small data.
\end{rem}

By \tref{thm: well-posedness and estimate of semi-implicit scheme}, for sufficiently small $\Delta t > 0$ and for all $n \le N := \lceil T_*/\Delta t \rceil$, we can construct a solution $u^n \in W \cap D(\partial\varphi) \, (n = 1, \dots, N)$ of \eref{eq: semi-implicit scheme}.
We introduce piecewise constant interpolations $u_{\Delta t}, \bar u_{\Delta t}$ and piecewise linear interpolations $\hat u_{\Delta t}, w_{\Delta t}$ of $\{u^n\}_{n=0}^N$ as follows:
\begin{alignat*}{2}
	u_{\Delta t}(t) &= u^n &&(n-1)\Delta t < t \le n\Delta t, \\
	\bar u_{\Delta t}(t) &= u^{n-1} && (n-1)\Delta t \le t < n\Delta t, \\
	\hat u_{\Delta t}(t) &= \frac{n\Delta t - t}{\Delta t} u^{n-1} + \frac{t - (n - 1)\Delta t}{\Delta t} u^n \qquad &&(n-1)\Delta t \le t \le n\Delta t, \\
	w_{\Delta t}(t) &= \begin{cases}
		u^1 & 0 \le t \le \Delta t, \\
		\hat u_{\Delta t}(t) \qquad & t \ge \Delta t,
	\end{cases}
\end{alignat*}
for $n = 1, 2, \dots, N$.
\tref{thm: well-posedness and estimate of semi-implicit scheme} tells us that $u_{\Delta t} \in L^\infty(0, T_*; V) \cap L^2(0, T_*; W)$, $\bar u_{\Delta t} \in L^\infty(0, T_*; V)$, $\hat u_{\Delta t} \in H^1(0, T_*; H) \cap L^\infty(0, T_*; V)$, and $w_{\Delta t} \in H^1(0, T_*; H) \cap L^\infty(0, T_*; V) \cap L^2(0, T_*; W)$ are uniformly bounded in $\Delta t$.
Therefore, there exist subsequences, denoted by the same symbols, such that
\begin{align*}
	u_{\Delta t} \rightharpoonup {}^\exists u &\quad\text{weakly in $L^2(0, T_*; W)$ and weakly-$*$ in $L^\infty(0, T_*; V)$}, \\
	\bar u_{\Delta t} \rightharpoonup {}^\exists \bar u &\quad\text{weakly-$*$ in $L^\infty(0, T_*; V)$}, \\
	\hat u_{\Delta t} \rightharpoonup {}^\exists \hat u &\quad\text{weakly in $H^1(0, T_*; H)$ and weakly-$*$ in $L^\infty(0, T_*; V)$}, \\
	w_{\Delta t} \rightharpoonup {}^\exists w &\quad\text{weakly in $H^1(0, T_*; H) \cap L^2(0, T_*; W)$ and weakly-$*$ in $L^\infty(0, T_*; V)$},
\end{align*}
as $\Delta t \to 0$.
By Aubin--Lions compactness theorem \cite[Corollary 4]{Sim1987}, $w_{\Delta t} \to w$ strongly in $C([0, T_*]; H) \cap L^2(0, T_*; V)$.

\begin{lem}
	With the setting above, we have $u = \bar u = \hat u = w$.
	Moreover, $u_{\Delta t}$, $\bar u_{\Delta t}$, and $\hat u_{\Delta t}$ converge to $u$ strongly in $L^2(0, T_*; V)$ as $\Delta t \to 0$.
\end{lem}
\begin{proof}
	Considering $u_{\Delta t}$ and $w_{\Delta t}$, we observe from direct computation and \tref{thm: well-posedness and estimate of semi-implicit scheme} that
	\begin{equation*}
		\|u_{\Delta t} - w_{\Delta t}\|_{L^2(0, T_*; V)}^2 \le \frac13 \sum_{n=1}^N \|u^n - u^{n-1}\|_{L^2(0, T_*; V)}^2 \Delta t \to 0 \quad (\Delta t \to 0).
	\end{equation*}
	Therefore,
	\begin{equation*}
		\|u_{\Delta t} - w\|_{L^2(0, T_*; V)} \le \|u_{\Delta t} - w_{\Delta t}\|_{L^2(0, T_*; V)} + \|w_{\Delta t} - w\|_{L^2(0, T_*; V)} \to 0 \quad (\Delta t \to 0),
	\end{equation*}
	thus $u_{\Delta t} \to w$ strongly (weakly as well) in $L^2(0, T_*; V)$.
	The uniqueness of a weak limit then implies $u = w$.
	By similar arguments that compare $\bar u_{\Delta t}$ and $\hat u_{\Delta t}$ with $w_{\Delta t}$, we also obtain $\bar u_{\Delta t} \to w$ and $\hat u_{\Delta t} \to w$ strongly in $L^2(0, T_*; V)$, and hence $\bar u = w = \hat u$.
\end{proof}

\begin{lem}
	For arbitrary $\tilde v \in L^2(0, T_*; V)$ we have
	\begin{equation} \label{eq: approximate VI integrated in time}
		\int_0^{T_*} \Big[ ( \partial_t \hat u_{\Delta t} + B(\bar u_{\Delta t}, u_{\Delta t}) - f_{\Delta t}, \tilde v - u_{\Delta t}) + a(u_{\Delta t}, \tilde v - u_{\Delta t}) + \varphi(\tilde v) - \varphi(u_{\Delta t}) \Big] \, dt \ge 0,
	\end{equation}
	where $f_{\Delta t}$ is a piecewise constant interpolation such that $f_{\Delta t}(t) = f^n$ for $t \in ((n-1) \Delta t, n \Delta t] \, (n = 1, \dots, N)$.
\end{lem}
\begin{proof}
	It suffices to show the following form of Jensen's inequality:
	\begin{equation} \label{eq: Jensen}
		\varphi \Big( \frac{1}{\Delta t} \int_{(n-1)\Delta t}^{n\Delta t} \tilde v \, dt \Big) \le \int_{(n-1)\Delta t}^{n\Delta t} \varphi(\tilde v) \, dt \quad (n = 1, \dots, N).
	\end{equation}
	In fact, we take $\frac{1}{\Delta t} \int_{(n-1)\Delta t(t)}^{n\Delta t} \tilde v \, dt$ as a test function in \eref{eq: semi-implicit scheme}, multiply the resulting inequality by $\Delta t$, apply \eref{eq: Jensen}, and add it for $n = 1, 2, \dots, N$ to obtain
	\begin{align*}
		&\sum_{n=1}^N \Big[ \Big( \frac{u^n - u^{n-1}}{\Delta t} + B(u^{n-1}, u^n) - f^n, \int_{(n-1)\Delta t}^{n\Delta t} \tilde v \, dt - u^n \Delta t \Big) \\
		&\qquad + a\Big( u^n, \int_{(n-1)\Delta t}^{n\Delta t} \tilde v \, dt - u^n \Delta t \Big) + \int_{(n-1)\Delta t}^{n\Delta t} \varphi(\tilde v) \, dt - \varphi(u^n) \Delta t \Big] \ge 0.
	\end{align*}
	This is equivalent to \eref{eq: approximate VI integrated in time} by virtue of piecewise constant structure of $\partial_t \hat u_{\Delta t}, \bar u_{\Delta t}, u_{\Delta t}$, and $f_{\Delta t}$.
	
	To prove \eref{eq: Jensen}, for $\lambda > 0$ let us introduce the Moreau regularization of $\varphi$ by
	\begin{equation*}
		\varphi_\lambda(v) = \min \Big\{ \varphi(z) + \frac{\|z - v\|_V^2}{2\lambda} \,\Big|\, z \in V \Big\} \quad (v \in V).
	\end{equation*}
	It follows from \cite[Theorem 2.9]{Bar2010} that $D(\varphi_\lambda) = D(\partial \varphi_\lambda) = V$, $\varphi_\lambda(v) \le \varphi(v)$ and $\varphi_\lambda(v) \to \varphi(v)$ as $\lambda \to 0$ for $v \in V$.
	Now, for $x := \frac{1}{\Delta t} \int_{(n-1)\Delta t}^{n\Delta t} \tilde v(s) \, ds \in V$ there exists some $\xi \in V'$ such that
	\begin{equation*}
		\left< \xi, y \right> - \left< \xi, x \right> + \varphi_\lambda(x) \le \varphi_\lambda(y) \qquad \forall y \in V.
	\end{equation*}
	Setting $y = \tilde v(t)$ and integrating the inequality above for $t \in [(n-1) \Delta t, n \Delta t]$, we see that the first two terms on the left-hand side are canceled and that
	\begin{equation*}
		\varphi_\lambda \Big( \frac{1}{\Delta t} \int_{(n-1)\Delta t}^{n\Delta t} \tilde v \, dt \Big) \le \int_{(n-1)\Delta t}^{n\Delta t} \varphi_\lambda(\tilde v) \, dt \le \int_{(n-1)\Delta t}^{n\Delta t} \varphi(\tilde v) \, dt.
	\end{equation*}
	We let $\lambda \to 0$ to get \eref{eq: Jensen}, which completes the proof.
\end{proof}
	
Now we can take the limit $\Delta t \to 0$ in \eref{eq: approximate VI integrated in time}.
For the term involving $B$, observe that
\begin{equation*}
	\int_0^{T_*} \big( B(\bar u_{\Delta t}, u_{\Delta t}), \tilde v - u_{\Delta t} \big) \, dt \to \int_0^{T_*} \big( B(u, u), \tilde v - u \big) \, dt,
\end{equation*}
In fact, since $\tilde v - u_{\Delta t} \to \tilde v - u$ strongly in $L^2(0, T_*; V)$ as $\Delta t \to 0$, it suffices to show $B(\bar u_{\Delta t}, u_{\Delta t}) \rightharpoonup B(u, u)$ weakly in $L^2(0, T_*; V')$.
For arbitrary $\psi \in L^2(0, T_*; V)$ we find from the bilinearity of $B$ and \eref{eq: boundedness of B} that
\begin{align*}
	\Big| \int_0^{T_*} \big( B(\bar u_{\Delta t}, u_{\Delta t}) - B(u, u), \psi \big) \, dt \Big| &\le C_B (\|\bar u_{\Delta t} - u\|_{L^2(0, T_*; V)} \|u_{\Delta t}\|_{L^\infty(0, T_*; V)} \\
		&\qquad + \|u\|_{L^\infty(0, T_*; V)} \|u_{\Delta t} - u\|_{L^2(0, T_*; V)}) \|\psi\|_{L^2(0, T_*; V)}
\end{align*}
converges to $0$ as $\Delta t \to 0$, which establishes the desired weak convergence.

As a conclusion,
\begin{equation*}
	\int_0^{T_*} \Big[ ( \partial_t u + B(u, u) - f, \tilde v - u) + a(u, \tilde v - u) + \varphi(\tilde v) - \varphi(u) \Big] \, dt \ge 0 \qquad \forall \tilde v \in L^2(0, T_*; V).
\end{equation*}
Then, an argument exploiting the Lebesgue differentiation theorem (see \cite[p.\ 57]{DuLi1976}) leads to
\begin{equation*}
	( \partial_t u(t) + B(u(t), u(t)) - f(t), v - u(t)) + a(u(t), v - u(t)) + \varphi(v) - \varphi(u(t)) \ge 0 \qquad \forall v \in V, \quad \text{a.e. } t \in (0, T_*),
\end{equation*}
which proves the existence part of \tref{main thm1}.

To show the uniqueness, let $u$ and $U$ be two solutions of \eref{eq: VI}.
We take $v = U$ and $v = u$ in the variational inequalities that $u$ and $U$ satisfy, respectively, and add the resulting two inequalities to obtain
\begin{align*}
	\frac12 \frac{d}{dt} \|u(t) - U(t)\|_H^2 + \|u(t) - U(t)\|_V^2 &\le - \big( B(u(t), u(t)) - B(U(t), U(t)), u(t) - U(t) \big) \\
		&\le \big| \big( B(u(t), u(t) - U(t)), u(t) - U(t) \big) \big| + \big| \big( B(u(t) - U(t), U(t)), u(t) - U(t) \big) \big| \\
		&\le C (\|u(t)\|_W + \|U(t)\|_W) \|u(t) - U(t)\|_H \|u(t) - U(t)\|_V \\
		&\le \frac12 \|u(t) - U(t)\|_V^2 + C^2 (\|u(t)\|_W^2 + \|U(t)\|_W^2) \|u(t) - U(t)\|_H^2
\end{align*}
for some constant $C > 0$, where we have used the bilinearity of $B$, (H3), and (H4).
Since $\int_0^{T_*} (\|u(t)\|_W^2 + \|U(t)\|_W^2) \, dt < \infty$, Gronwall's inequality concludes $u(t) = U(t)$ in $H$ for $0 \le t \le T_*$.
This completes the proof of \tref{main thm1}.

\section{Proof of \tref{main thm2}} \label{sec6}
In view of \tref{main thm1}, it remains to show additional regularity of the solution $u$ of \eref{eq: VI}.
For almost every $h > 0$, we take $v = u(t)$ (resp.\ $v = u(t + h)$) in the variational inequality that $u(t + h)$ (resp.\ $u(t)$) satisfies, and add the resulting inequalities.
It follows that
\begin{align*}
	\frac12 \frac{d}{dt} \|u(t + h) - u(t)\|_H^2 + \|u(t + h) - u(t)\|_V^2
		&\le \big( f(t + h) - f(t), u(t + h) - u(t) \big) \\
		& \quad + \big( B( u(t + h), u(t + h) ) - B( u(t), u(t) ), u(t + h) - u(t) \big).
\end{align*}
The first term on the right-hand side is bounded by $\|f(t + h) - f(t)\|_H \|u(t + h) - u(t)\|_H$.
We observe from the bilinearlity of $B$ and from $\|B(u, v)\|_H \le C_B' \|u\|_V \|v\|_W$ (by (H4)) that the second term equals
\begin{align*}
	&\Big( B(u(t + h), u(t + h) - u(t)) + B(u(t + h) - u(t), u(t)), u(t + h) - u(t) \Big) \\
		\le \; &C_{\theta_1, 1/4} \|u(t + h)\|_V^{\theta_1} \|u(t + h) - u(t)\|_H^2 + \frac14 \|u(t + h) - u(t)\|_V^2 \\
		&\qquad + C_B'^2 \|u(t)\|_W^2 \|u(t + h) - u(t)\|_H^2 + \frac14 \|u(t + h) - u(t)\|_V^2.
\end{align*}
Consequently,
\begin{equation} \label{eq4: proof of main thm2}
\begin{aligned}
	\frac12 \frac{d}{dt} \|u(t + h) - u(t)\|_H^2 + \frac12 \|u(t + h) - u(t)\|_V^2 &\le \|f(t + h) - f(t)\|_H \|u(t + h) - u(t)\|_H \\
		&\qquad + \big( C_{\theta_1, 1/4} \|u(t + h)\|_V^{\theta_1} + C_B'^2 \|u(t)\|_W^2 \big) \|u(t + h) - u(t)\|_H^2
\end{aligned}
\end{equation}
for almost every $h \in (0, T_*)$ and $t \in (0, T_* - h)$; in particular,
\begin{equation} \label{eq1: proof of main thm2}
	\frac{d}{dt} \|u(t + h) - u(t)\|_H \le \|f(t + h) - f(t)\|_H + \big( C_{\theta_1, 1/4} \|u(t + h)\|_V^{\theta_1} + C_B'^2 \|u(t)\|_W^2 \big) \|u(t + h) - u(t)\|_H.
\end{equation}
By Gronwall's inequality, for all $h \in [0, T_*]$ and $t \in [0, T_* - h]$ one has
\begin{equation} \label{eq2: proof of main thm2}
	\|u(t + h) - u(t)\|_H \le \|u(h) - u^0\|_H + \exp(C_{\theta_1, 1/4} T_* \|u\|_{L^\infty(0, T_*; V)}^{\theta_1} + C_B'^2 \|u\|_{L^2(0, T_*; W)}^2) \int_0^t \|f(s + h) - f(s)\|_H \, ds
\end{equation}

On the other hand, the assumption $(A + \partial\varphi)(u^0) \cap H \neq \emptyset$, together with $B(u^0, u^0) \in H$, implies existence of some $u^0_* \in H$ such that
\begin{equation*}
	a(u^0, v - u^0) + \varphi(v) - \varphi(u^0) \ge (- u^0_* - B(u^0, u^0), v - u^0) \quad \forall v \in V.
\end{equation*}
Take $v = u(t)$ above and $v = u^0$ in the variational inequality that $u(t)$ satisfies, that is,
\begin{equation*}
	\big( \partial_t (u(t) - u^0), v - u(t) \big) + a(u(t), v - u(t)) + \varphi(v) - \varphi(u(t)) \ge \big( f(t) - B(u(t), u(t)), v - u(t) \big) \quad \forall v \in V.
\end{equation*}
We add the resulting inequalities to obtain, after a similar calculation to \eref{eq1: proof of main thm2},
\begin{equation*}
	\frac{d}{dt} \|u(t) - u^0\|_H \le \|f(t) - u^0_*\|_H + \big( C_{\theta_1, 1/4} \|u(t)\|_V^{\theta_1} + C_B'^2 \|u^0\|_W^2 \big) \|u(t) - u^0\|_H \quad \text{for a.e.\ } t \in (0, T),
\end{equation*}
which, by Gronwall's inequality, implies
\begin{equation} \label{eq3: proof of main thm2}
	\|u(h) - u^0\|_H \le \exp \big( h(C_{\theta_1, 1/4} \|u\|_{L^\infty(0, T_*; V)}^{\theta_1} + C_B'^2 \|u^0\|_W^2) \big) \int_0^h \|f(s) - u^0_*\|_H \, ds \quad \forall h \in [0, T_*].
\end{equation}

Substituting \eref{eq3: proof of main thm2} into \eref{eq2: proof of main thm2} and recalling $f \in W^{1,1}(0, T; H)$, we deduce that $u : [0, T] \to H$ is Lipschitz continuous, and hence $u \in W^{1, \infty}(0, T_*; H)$.
Then, integrating \eref{eq4: proof of main thm2} for $0 \le t \le T_* - h$, dividing the both sides by $h^2$, and using \cite[Proposition A.7 and Corollaire A.2]{Bre1973}, we obtain $u \in H^1(0, T_*; V)$.

Now we claim that $B(u, u) \in W^{1, 1}(0, T_*; H)$.
In fact, $B(u, u) \in L^1(0, T_*; H)$ because $\int_0^{T_*} \|B(u, u)\|_H \, dt \le C \|u\|_{L^2(0, T_*; V)} \|u\|_{L^2(0, T_*; W)}$.
By the bilinearity of $B$, we have $\partial_t (B(u, u)) = B(\partial_t u, u) + B(u, \partial_t u)$, so that
\begin{equation*}
	\int_0^{T_*} \|\partial_t [B(u(t), u(t))] \|_H \, dt \le C \int_0^{T_*} \|\partial_t u(t)\|_V \|u(t)\|_W \, dt \le
		C \|\partial_t u(t)\|_{L^2(0, T_*; V)} \|u\|_{L^2(0, T_*; W)}
\end{equation*}
for some constant $C > 0$, where we have used (H3) and (H4).
In particular, $B(u, u)$ belongs to $L^\infty(0, T_*; H)$.

Applying the regularity hypothesis (H5) to
\begin{equation*}
	a(u(t), v - u(t)) + \varphi(v) - \varphi(u(t)) \ge (f(t) - \partial_t u(t) - B(u(t), u(t)), v - u(t)) \qquad \forall v \in V,
\end{equation*}
we get $u \in L^\infty(0, T_*; W)$.

Finally, in order to see that \eref{eq: VI} holds everywhere (rather than a.e.\ sense), we rewrite \eref{eq: VI} as
\begin{equation} \label{eq: arranged VI}
	\partial_t u(t) + (A + \partial\varphi)u(t) \ni f(t) - B(u(t), u(t)),
\end{equation}
where $A + \partial\varphi : H \to 2^H$ is maximal monotone (recall the argument before \lref{lem: existence for Oseen VI}) and the right-hand side is in $W^{1, 1}(0, T; H)$.
We then find from \cite[Theorem 4.6]{Bar2010} that $u: [0, T_*) \to H$ is differentiable from the right and that \eref{eq: arranged VI} holds for all $t \in [0, T_*)$.
For the sake of completeness, we provide a proof of this fact in the appendix.
This completes the proof of \tref{main thm2}.

\appendix
\section{Remarks on solution in Kiselev--Ladyzhenskaya class}
Let $H$ be a Hilbert space with the inner product $(\cdot, \cdot)$ and $A : H \to 2^H$ be a maximal monotone operator with the domain $D(A) \subset H$, which can be identified with a subset $A \subset H \times H$.
We consider the following differential inclusion:
\begin{equation} \label{eq: DI}
	\partial_t u(t) + A u(t) \ni f(t).
\end{equation}
Supposing $u : (0, T) \to D(A)$ satisfies the above relation, we find from the maximal monotonicity that $Au(t)$ is a nonempty closed convex subset of $H$ for a.e.\ $t \in (0, T)$, and so is $f(t) - Au(t)$.
Therefore, the minimal section $(f(t) - Au(t))^0 := \operatorname{argmin} \{ \|\xi\|_H \mid \xi \in f(t) - Au(t) \}$ is well defined and is a singleton (cf.\ \cite[p.\ 101]{Bar2010}).

The following lemma and two propositions are essentially taken from \cite[Proposition 3.3]{Bre1973}.
\begin{lem}
	Assume that $f \in C([0, T]; H)$ and that $u \in W^{1, \infty}(0, T; H)$ satisfies $u(t) \in D(A)$ and \eref{eq: DI} for a.e.\ $t \in (0, T)$.
	Then, for all $0 \le t_0 \le t \le T$ and $[x, y] \in A$ we we have
	\begin{equation} \label{eq: result1 of Lemma A.1}
		( u(t) - u(t_0), u(t_0) - x ) \le \int_{t_0}^t ( f(s) - y, u(s) - x ) \, ds.
	\end{equation}
	If in addition $u(t_0) \in D(A)$, then for all $h \in [0, T - t_0]$ we have
	\begin{equation} \label{eq: result2 of Lemma A.1}
		\|u(t_0 + h) - u(t_0)\|_H \le \int_{t_0}^{t_0 + h} \| f(s) - f(t_0) + ( f(t_0) - Au(t_0) )^0 \|_H \, ds.
	\end{equation}
\end{lem}
\begin{proof}
	Because $(a - b, b) \le (\|a\|_H^2 - \|b\|_H^2)/2$, the left-hand side of \eref{eq: result1 of Lemma A.1} is $\le \frac12 \|u(t) - x\|_H^2 - \frac12 \|u(t_0) - x\|_H^2$.
	On the other hand, multiplying $\partial_t u(t) + \xi(t) = f(t)$ by $u(t) - x$, where $\xi(t) \in Au(t)$, we have
	\begin{equation} \label{eq1: proof of Lemma A.1}
		\frac12 \frac{d}{dt} \|u(t) - x\|_H^2 + \underbrace{(\xi(t) - y, u(t) - x)}_{ \ge 0 } = (f(t) - y, u(t) - x) \quad \text{for a.e. } t \in (0, T).
	\end{equation}
	Integration of this for $t_0 \le s\le t$ leads to \eref{eq: result1 of Lemma A.1}.
	In case $u(t_0) \in D(A)$, we set $x = u(t_0)$ in \eref{eq1: proof of Lemma A.1} and divide by $\|u(t) - u(t_0)\|_H$ to obtain
	\begin{equation*}
		\frac{d}{dt} \|u(t) - u(t_0)\|_H \le \|f(t) - f(t_0) + f(t_0) - y\|_H \qquad \forall y \in A(u(t_0)), \quad \text{for a.e. } t \in (0, T).
	\end{equation*}
	Choosing $y$ in such a way that $f(t_0) - y = (f(t_0) - Au(t_0))^0$ and integrating with respect to $t$ yield \eref{eq: result2 of Lemma A.1}.
\end{proof}

\begin{prop} \label{prop A.1}
	Assume that $f \in C([0, T]; H)$ and that $u \in W^{1, \infty}(0, T; H)$ satisfies \eref{eq: DI} for a.e.\ $t \in (0, T)$.
	Then, for all $t \in [0, T)$ we have $u(t) \in D(A)$.
	Moreover, $u: [0, T) \to H$ is right-differentiable, that is, $\frac{d^+}{dt} u(t_0) := \lim_{h \downarrow 0} \frac{u(t_0 + h) - u(t_0)}{h}$ exists in $H$ for all $0 \le t_0 < T$; actually, $\frac{d^+}{dt} u(t_0) = (f(t_0) - Au(t_0))^0$ holds.
\end{prop}
\begin{proof}
	Since $u \in W^{1, \infty}(0, T; H)$ implies $\|\frac{u(t_0 + h) - u(t_0)}{h}\|_H \le C$ for some constant $C$ independent of $t_0 \in [0, T)$ and of $h \in (0, T - t_0]$, there exist a subsequence $h' \to 0$ and some $\xi \in H$ such that $\frac{u(t_0 + h') - u(t_0)}{h'} \rightharpoonup \xi$ weakly in $H$.
	By \eref{eq: result1 of Lemma A.1} we have
	\begin{equation*}
		\Big( \frac{u(t_0 + h') - u(t_0)}{h'}, u(t_0) - x \Big) \le \frac1{h'} \int_{t_0}^{t_0 + h'} (f(s) - y, u(s) - x) \, ds \qquad \forall [x, y] \in A.
	\end{equation*}
	Taking the limit $h' \to 0$ and using the Lebesgue differentiation theorem, we obtain
	\begin{equation*}
		(\xi, u(t_0) - x) \le (f(t_0) - y, u(t_0) - x) \Longleftrightarrow (f(t_0) - \xi - y, u(t_0) - x) \ge 0 \qquad \forall [x, y] \in A,
	\end{equation*}
	which combined with the maximal monotonicity of $A$ implies that $u(t_0) \in D(A)$ and $f(t_0) - \xi \in Au(t_0)$, i.e., $\xi \in f(t_0) - Au(t_0)$.
	
	Now it follows from lower-semicontinuity of a norm with respect to weak convergence, \eref{eq: result2 of Lemma A.1} divided by $h$, and another use of the Lebesgue differentiation theorem, that
	\begin{equation*}
		\|\xi\|_H \le \liminf_{h' \downarrow 0} \Big\| \frac{u(t_0 + h') - u(t_0)}{h'} \Big\|_H \le
			\limsup_{h \downarrow 0} \Big\| \frac{u(t_0 + h) - u(t_0)}{h} \Big\|_H \le \|(f(t_0) - Au(t_0))^0\|_H.
	\end{equation*}
	By the property of the minimal section we must have $\xi = (f(t_0) - Au(t_0))^0$, which results in $\Big\| \frac{u(t_0 + h) - u(t_0)}{h} \Big\|_H \to \|\xi\|_H$ and in $\frac{u(t_0 + h) - u(t_0)}{h} \rightharpoonup \xi$ with respect to the whole sequence $h \to 0$.
	The proposition is concluded by the well-known fact that a weakly converging sequence in a Hilbert space is strongly convergent if its norm converges.
\end{proof}

Finally let us prove the right-continuity of of the right derivative.
\begin{prop}
	In addition to the hypotheses of \pref{prop A.1}, assume that $f \in W^{1,1}(0, T; H)$.
	Then $\frac{d^+}{dt} u: [0, T) \to H$ is right-continuous, that is, $\frac{d^+}{dt} u(t) \to \frac{d^+}{dt} u(t_0)$ as $t \downarrow t_0$ for all $0 \le t_0 < T$.
\end{prop}
\begin{proof}
	We multiply $\partial_t u(t + h) + A u(t + h) \ni f(t + h)$ and $\partial_t u(t) + A u(t) \ni f(t)$ by $u(t+h) - u(t)$ and $u(t) - u(t + h)$, respectively, and add the resulting two inclusions.
	It follows that, for some $y \in Au(t)$ and $z \in Au(t_0)$,
	\begin{equation*}
		\frac12 \frac{d}{dt} \|u(t + h) - u(t)\|_H^2 + \underbrace{(y - z, u(t) - u(t_0))}_{ \ge 0 } = (f(t) - f(t_0), u(t) - u(t_0)) \qquad \text{for a.e. } t \in (0, T), \; h \in (0, T - t).
	\end{equation*}
	Dividing the both sides by $\|u(t + h) - u(t)\|_H$ and integrating for $t_0 \le s \le t$, we have
	\begin{equation*}
		\|u(t+h) - u(t)\|_H \le \|u(t_0+h) - u(t_0)\|_H + \int_{t_0}^t \|f(s+h) - f(s)\|_H \, ds \qquad \text{for all} \; t \ge t_0 \ge 0, \; h \in [0, T - t].
	\end{equation*}
	Dividing by $h$, taking the limit $h \to 0$, and using the dominated convergence theorem, we deduce that
	\begin{equation*}
		\Big\| \frac{d^+}{dt}u(t) \Big\|_H \le \Big\| \frac{d^+}{dt}u(t_0) \Big\|_H + \int_{t_0}^t \|\partial_s f\|_H \, ds \qquad \text{for all} \; t \ge t_0 \ge 0.
	\end{equation*}
	
	We now let $t \downarrow t_0$.
	Since the right-hand side of the above estimate can be bounded uniformly in $t$, there exist a subsequence $t' \downarrow t_0$ and some $\xi \in H$ such that $\frac{d^+}{dt}u(t') \rightharpoonup \xi$ weakly in $H$ as $t' \downarrow t_0$.
	It then follows that
	\begin{equation*}
		\|\xi\|_H \le \liminf_{t' \downarrow t_0} \Big\| \frac{d^+}{dt}u(t') \Big\|_H \le \limsup_{t \downarrow t_0} \Big\| \frac{d^+}{dt}u(t) \Big\|_H \le
			\Big\| \frac{d^+}{dt}u(t_0) \Big\|_H = \| (f(t_0) - Au(t_0))^0 \|_H.
	\end{equation*}
	Observing that $u(t') \to u(t_0)$ strongly in $H$ and that $Au(t') \ni f(t') - \frac{d^+}{dt}u(t') \rightharpoonup f(t_0) - \xi$ weakly in $H$, we find from the demi-closedness of $A$ (see \cite[Proposition 3.4]{Bar2010}) that $f(t_0) - \xi \in Au(t_0)$, i.e., $\xi \in f(t_0) - Au(t_0)$.
	By the property of the minimal section we have $\xi = (f(t_0) - Au(t_0))^0 = \frac{d^+}{dt}u(t_0)$.
	The proof can be completed similarly to the last argument in the proof of \pref{prop A.1}.
\end{proof}


\end{document}